\documentclass
{amsart}

\makeatletter
\def\tbcaption{\def\@captype{table}\caption}
\def\figcaption{\def\@captype{figure}\caption}
\makeatother

\usepackage[dvipdfm]{graphicx}

\usepackage{amssymb,amsmath,amsthm,amscd,amsfonts}

\newtheorem{main}{Main Theorem}

\newtheorem{theorem}{Theorem}[section]
\newtheorem{proposition}[theorem]{Proposition}
\newtheorem{corollary}[theorem]{Corollary}

\newtheorem{definition}[theorem]{Definition}
\newtheorem{remark}[theorem]{Remark}

\begin{document}
\title[]
{Weighted Hamiltonian stationary Lagrangian submanifolds and generalized Lagrangian mean curvature flows in toric almost Calabi--Yau manifolds}
\author{Hikaru Yamamoto}
\address{Graduate School of Mathematical Sciences, The University of Tokyo, 3-8-1 Komaba Meguro-ku Tokyo 153-8914, Japan}
\email{yamamoto@ms.u-tokyo.ac.jp}
\begin{abstract} 
In this paper we generalize examples of Hamiltonian stationary Lagrangian submanifolds constructed by Lee and Wang \cite{LeeWang} in $\mathbb{C}^{m}$ to 
toric almost Calabi--Yau manifolds. 
We construct examples of weighted Hamiltonian stationary Lagrangian submanifolds in toric almost Calabi--Yau manifolds 
and solutions of generalized Lagrangian mean curvature flows starting from these examples. 
We allow these flows to have some singularities and topological changes. 
\end{abstract}
\thanks{This work was supported by Grant-in-Aid for JSPS Fellows Grant Number 25$\cdot$6407 and the Program for Leading Graduate Schools, MEXT, Japan. }
\maketitle
\section{Introduction}\label{Intro}
Recently study of Lagrangian submanifolds acquire much importance in association with Mirror Symmetry. 
There are several classes of Lagrangian submanifolds. 
For example, special Lagrangian submanifolds are defined in Calabi--Yau manifolds by Harvey and Lawson in \cite{HarveyLawson} 
and they have an important role in the Strominger--Yau--Zaslow conjecture \cite{StromingerYauZaslow}. 
A class of Hamiltonian stationary Lagrangian submanifolds is also defined in Calabi--Yau manifolds, 
especially a special Lagrangian submanifold is a Hamiltonian stationary Lagrangian submanifold. 
In general constructing explicit examples of special or Hamiltonian stationary Lagrangian submanifolds is difficult. 
However some examples are constructed in the case that the ambient Calabi--Yau manifold has symmetries, especially in $\mathbb{C}^{m}$. 

For example, one of examples of special Lagrangian submanifolds in $\mathbb{C}^{m}$ constructed by Harvey and Lawson in \cite[III.3.A]{HarveyLawson} is defined by 
$$M_{c}:=\{\,(z_{1},\dots,z_{m})\in\mathbb{C}^{m}\mid \mathrm{Re}(z_{1}\cdots z_{m})=c_{1},\, |z_{1}|^{2}-|z_{j}|^{2}=c_{j}\, (j=2,\dots,m) \}, $$
where $c=(c_{1},\dots,c_{m})\in\mathbb{R}^{m}$ and note that the phase of $M_{c}$ is $i^{m}$. 
We remark that if we put $z_{j}=x_{j}e^{i\theta_{j}}$ for $x_{j}\in\mathbb{R}$ then $M_{c}$ is written by 
$$\{\,\exp(s_{2}\zeta_{2}+\dots+s_{n}\zeta_{m})\cdot x \in\mathbb{C}^{m}\mid x\in\mathbb{R}^{m},\, s_{j}\in\mathbb{R},\, \langle\mu(x),\zeta_{j}\rangle=\frac{c_{j}}{2}\, (j=2,\dots,m) \}, $$
where $\zeta_{j}:=(1,0,\dots,0,-1,0,\dots,0)=e_{1}-e_{j}\in\mathbb{R}^{m}$ and $\mu(x):=\frac{1}{2}(x_{1}^{2},\dots,x_{m}^{2})$ 
and we define $\exp(v)\cdot x=(x_{1}e^{2\pi iv_{1}},\dots,x_{m}e^{2\pi iv_{m}})$ for $v=(v_{1},\dots,v_{m})\in\mathbb{R}^{m}$. 
This is a $T^{m-1}$-invariant special Lagrangian submanifold in $\mathbb{C}^{m}$. 

Next, one of examples of special Lagrangian submanifolds in $\mathbb{C}^{m}$ constructed by Joyce in \cite[Example 9.4]{Joyce} is defined by 
$$N_{c}^{a_{1},\dots,a_{m}}:=\{\,(x_{1}e^{2\pi ia_{1}\theta},\dots,x_{m}e^{2\pi ia_{m}\theta})\in\mathbb{C}^{m}\mid \theta\in\mathbb{R},\,a_{1}x_{1}^{2}+\dots+a_{m}x_{m}^{2}=c \}, $$
where $a=(a_{1},\dots,a_{m})\in\mathbb{R}^{m}$. 
Joyce constructed this example by using a moment map of $T^{1}$-action on $\mathbb{C}^{m}$. 
Of course, in the same way as $M_{c}$, $N_{c}^{a_{1},\dots,a_{m}}$ is written by 
$$\{\,\exp(a\theta)\cdot x \mid x\in\mathbb{R}^{m},\,\theta\in\mathbb{R},\,\langle\mu(x),a\rangle=\frac{c}{2}\,\}. $$
This is a $T^{1}$-invariant special Lagrangian submanifold in $\mathbb{C}^{m}$. 

These two examples suggest that a torus action, a real structure and a moment map are useful to construct special Lagrangian submanifolds. 
From this view point, the author generalized Joyce's example $N_{c}^{a_{1},\dots,a_{m}}$ in $\mathbb{C}^{m}$ to 
in an $m$-dimensional toric almost Calabi--Yau cone manifold in \cite{Yamamoto}. 
That is, the author constructed examples of special Lagrangian submanifolds of the form 
$$\{\,\exp(t\zeta)\cdot p \mid p\in M^{\sigma},\, t\in\mathbb{R},\, \langle\mu(p),\zeta\rangle=c \,\}$$
in almost Calabi--Yau cone manifolds $(M,\omega,g,J)$, where $M^{\sigma}$ is the real form of $M$, $\mu$ is a moment map of $T^{m}$-action on $M$, 
$\zeta$ is a vector in $\mathbb{R}^{m}$ satisfying a special condition and $c$ is a constant. 
This is a $T^{1}$-invariant special Lagrangian submanifold in almost Calabi--Yau cone manifolds $(M,\omega,g,J)$. 

This type of constructions is also effective to construct examples of Hamiltonian stationary Lagrangian submanifolds. 
Actually in $\mathbb{C}^{m}$ Lee and Wang \cite{LeeWang} proved that $V_{t}$ defined by 
\begin{align*}
\biggl\{\, (x_{1}e^{2\pi i\zeta_{1}s}&,\dots,x_{m}e^{2\pi i\zeta_{m}s})  \in \mathbb{R}^{m} \,\bigg|\, 0\leq s \leq 1, \\
 &\sum_{j=1}^{m}\zeta_{j}x^{2}_{j}=-4\pi t\sum_{j=1}^{m}\zeta_{j},\, x=(x_{1},\dots,x_{m})\in\mathbb{R}^{m}\,\biggr\}
\end{align*}
is a Hamiltonian stationary Lagrangian submanifold for all $\zeta\in\mathbb{R}^{m}$ and $c\in\mathbb{R}$. 
Furthermore they proved that this family $\{V_{t}\}_{t\in\mathbb{R}}$ is a solutions of Brakke flow. 
Here Brakke flow is a weak formulation of a mean curvature flow with singularities proposed by Brakke in \cite{Brakke}. 

To get a special Lagrangian submanifold in a given Calabi--Yau manifold, 
a mean curvature flow is one of potential approaches since 
if a long time solution of a mean curvature flow starting from a fixed Lagrangian submanifold exists 
and converges to a smooth manifold then it is a minimal Lagrangian submanifold that is a special Lagrangian submanifold. 
However a mean curvature flow does not have a long time solution in general and it develops singularities. 
Thus it is meaningful to construct examples of Lagrangian mean curvature flows with singularities to understand the motion of Lagrangian mean curvature flows 
and to develop this strategy. 

In this paper we construct explicit examples of special or weighted Hamiltonian stationary Lagrangian submanifolds in toric almost Calabi--Yau manifolds and 
construct solutions of generalized Lagrangian mean curvature flows with singularities and topological changes starting from these examples. 
These examples can be considered as some kind of generalization of examples of Lee and Wang \cite{LeeWang} in $\mathbb{C}^{m}$ to toric almost Calabi--Yau manifolds. 
When the ambient space is a general toric almost Calabi--Yau manifold then its topology is not simple and there are many fixed points of torus action 
hence we can get examples of special or weighted Hamiltonian stationary Lagrangian submanifolds with various topologies 
and its generalized Lagrangian mean curvature flow develops singularities many times 
though examples of Lee and Wang in $\mathbb{C}^{m}$ develops a singularity once. 
In this paper we use notions of {\it weighted} Hamiltonian stationary and {\it generalized} Lagrangian mean curvature flow. 
These notions are modifications of the ordinary notions of Hamiltonian stationary and Lagrangian mean curvature flow defined in Calabi--Yau manifolds, 
since we use almost Calabi--Yau manifolds rather than Calabi--Yau manifolds, see Section \ref{angle} for precise definitions. 

We now give a description of the main results of this paper. 
Let $(M,\omega,g,J, \Omega_{\gamma})$ be a real $2m$-dimensional toric almost Calabi--Yau manifold with torus $T^{m}$ action that is a toric K\"ahler manifold 
with a nonvanishing holomorphic $(m,0)$-form $\Omega_{\gamma}$ defined by a vector $\gamma$ in $\mathbb{Z}^{m}$, 
here $\gamma$ is canonically determined by the toric structure of $(M,J)$, 
see Section \ref{angle} for the definition of $\Omega_{\gamma}$. 
Note that we do not assume that $(M,\omega,g,J)$ is Ricci-flat. 
Since $(M,\omega,g,J)$ is a toric K\"ahler manifold there exist 
a moment map $\mu:M\rightarrow \Delta$ with a moment polytope $\Delta$ 
and an anti-holomorphic and anti-symplectic involution $\sigma:M\rightarrow M$, see Section \ref{TKm} for more precisely settings. 
We denote the fixed point set of $\sigma$ by $M^{\sigma}$ and call it the real form of $M$. 
This is a real $m$-dimensional submanifold in $M$. 
Fix an integer $n$ with $0 \leq n \leq m$. 
Take a set of $n$ vectors $\zeta=\{\,\zeta_{1},\dots,\zeta_{n}\,\}\subset\mathbb{Z}^{m}$ and a set of $n$ constants $c=\{\,c_{1},\dots,{c}_{n}\,\}\subset\mathbb{R}$ 
and consider the set 
$$M^{\sigma}_{\zeta,c}:=\{\, p\in M^{\sigma} \mid \langle\mu(p), \zeta_{i}\rangle=c_{i},\, i=1\dots,n \,\} .$$
We assume that $M^{\sigma}_{\zeta,c}$ is a real $(m-n)$-dimensional submanifold in $M^{\sigma}$ and 
$T_{\zeta}:=V_{\zeta}/(V_{\zeta}\cap\mathbb{Z}^{m})$ is isomorphic to a subtorus $T^{n}$ of $T^{m}$, where $V_{\zeta}:=\mathrm{Span}_{\mathbb{R}}\{\zeta_{1},\dots,\zeta_{n}\}$. 
Then we put a real $m$-dimensional manifold as 
\begin{align}\label{form}
L_{\zeta,c}:=M^{\sigma}_{\zeta,c}\times T_{\zeta}
\end{align}
and define a map $F_{\zeta,c}:L_{\zeta,c}\rightarrow M$ by 
$$F_{\zeta,c}(p,[v]):=\exp v\cdot p. $$
\begin{main}
$F_{\zeta,c}:L_{\zeta,c}\rightarrow M$ is a Lagrangian immersion and its Lagrangian angle $\theta_{\zeta,c}:L_{\zeta,c}\rightarrow \mathbb{R}/\pi\mathbb{Z}$ is 
given by $\theta_{\zeta,c}(p,[v])=2\pi\langle \gamma,v\rangle + \frac{\pi}{2}n\,(\mathrm{mod}.\,\pi)$ 
and $F_{\zeta,c}:L_{\zeta,c}\rightarrow M$ is a $T^{n}$-invariant weighted Hamiltonian stationary Lagrangian submanifold for all $\zeta$ and $c$. 
\end{main}

\begin{corollary}
$F_{\zeta,c}:L_{\zeta,c}\rightarrow M$ is a special Lagrangian submanifold if and only if 
$\langle\gamma,\zeta_{i}\rangle=0$ for all $i=1,\dots,n$. 
\end{corollary}

\begin{main}
The family of the images of $\{F_{\zeta,c(t)}:L_{\zeta,c(t)}\rightarrow M\}_{0\leq t\leq T}$ is a solution of 
generalized Lagrangian mean curvature flow with singularities and topological changes with initial condition $F_{\zeta,c}$, 
where $c(t):=\{\, c_{1}(t),\dots,c_{n}(t)\,\}$ and each $c_{j}(t)$ is given by $c_{j}(t):=c_{j}-2\pi t\langle\gamma,\zeta_{j}\rangle$. 
Here $T$ is the first time that $M^{\sigma}_{\zeta,c(t)}$ becomes empty. 
\end{main}

The definition of Lagrangian angle is given in Section \ref{angle} 
and the notion of generalized Lagrangian mean curvature flow with singularities and topological changes is defined in Section \ref{MCF}. 
Roughly speaking, this flow is parametrized by a smooth flow except some $m$-dimensional Hausdorff measure zero sets. 

We note that the example $M_{c}$ of Harvey and Lawson is the case when $n=m-1$,  
and $N_{c}^{a_{1},\dots,a_{m}}$ of Joyce, $V_{t}$ of Lee and Wang and the previous work of the author in \cite{Yamamoto} are the case when $n=1$. 
After finishing my work, I learned from H. Konno that the Mironov and Panov \cite{MironovPanov} constructed 
examples of $T^{n}$-invariant Hamiltonian stationary Lagrangian submanifolds in $m$-dimensional toric varieties for $0\leq n\leq m$. 
First Mironov \cite{Mironov} constructed $T^{n}$-invariant Hamiltonian stationary or minimal Lagrangian submanifolds in $\mathbb{C}^{m}$ and $\mathbb{CP}^{m}$.  
These examples can be written as the form (\ref{form}) in $\mathbb{C}^{m}$. 
In \cite{MironovPanov}, they used a K\"ahler quotient of $\mathbb{C}^{m}$ to construct new examples in toric varieties. 
Our method is different from theirs in that we use the real form and a moment map rather than K\"ahler quotient to construct examples and 
furthermore we study motion of generalized Lagrangian mean curvature flows starting from these examples. 
 \subsection*{Acknowledgements}
 I would like to thank to Professor Akito Futaki for his comments and I also thank to 
 Professor Hiroshi Konno for letting me know the work of Mironov and Panov. 
\section{Toric K\"ahler manifold}\label{TKm}
Let $T^{m}\cong (S^{1})^{m}$ be an $m$-dimensional real torus and $(M,\omega, g, J)$ be a toric K\"ahler manifold with complex dimension $m$. 
Then $T^{m}$ acts on $M$ effectively and the K\"ahler form $\omega$ is invariant under the action. 
Let $\mu:M\rightarrow \mathfrak{g}^{*}$ be a moment map and $\Delta:=\mu(M)$ be a moment polytope, 
where $\mathfrak{g}$ is a Lie algebra of $T^{m}$ and $\mathfrak{g}^{*}$ is its dual. 
Since $(M,J)$ is a toric variety, there is a complex torus $T^{m}_{\mathbb{C}}\cong (\mathbb{C}^{\times})^{m}$ 
which is a complexification of $T^{m}$ and $T^{m}_{\mathbb{C}}$ acts on $(M,J)$ as biholomorphic automorphisms. 
Then $M$ has an open dense $T^{m}_{\mathbb{C}}$-orbit and we denote the  fan of $(M,J)$ by $\Sigma$. 
Let $\Sigma(1):=\{\, \rho\in\Sigma\mid\dim\rho=1\, \}$ be a set of 1-dimensional cones in $\Sigma$. 
We assume that $\Sigma(1)$ is a finite set and write $\Sigma(1)=\{\rho_{1},\dots,\rho_{d}\}$. 
Let $\lambda_{i}$ be the primitive element that generates $\rho_{i}$ for $i=1,\dots,d$, 
that is, $\rho_{i}=\mathbb{R}^{+}\lambda_{i}$. 
Note that, in general, $\Delta$ is not a closed subset in $\mathfrak{g}^{*}$. 
For example, if we consider a toric K\"ahler manifold constructed by removing all fixed points of torus action from some toric K\"ahler manifold, 
then its moment polytope has a shape that all vertices are removed from the original polytope and this is not a closed subset. 

We assume that there exist $\kappa_{i}$ in $\mathbb{R}$ for $i=1,\dots,d$ so that the closure of $\Delta$ is given by
\begin{align*}
\overline{\Delta}=\bigcap_{i=1}^{d}H_{\lambda_{i},\kappa_{i}}^{+}.
\end{align*} 
Here for a nonzero vector $\lambda$ in $\mathfrak{g}$ and $\kappa$ in $\mathbb{R}$, 
we define the affine hyperplane $H_{\lambda,\kappa}$ and closed half-space $H_{\lambda,\kappa}^{+}$ by
\begin{align*}
H_{\lambda,\kappa}:=\{\, y\in\mathfrak{g}^{*}\mid\langle y, \lambda\rangle=\kappa\, \}\quad \mathrm{and} \quad 
H_{\lambda,\kappa}^{+}:=\{\, y\in\mathfrak{g}^{*}\mid\langle y, \lambda\rangle\geq \kappa\, \}. 
\end{align*}
A subset $F\subset\overline{\Delta}$ is called a face of $\overline{\Delta}$ if and only if 
there exist a vector $v$ in $\mathfrak{g}$ and a constant $c$ such that 
\begin{align*}
\overline{\Delta}\subset H_{v,c}^{+}\quad\mathrm{and}\quad F=\overline{\Delta}\cap H_{v,c}. 
\end{align*}
We denote the set of all faces of $\overline{\Delta}$ by $\mathcal{F}$. 
Then there exists a subset $\mathcal{G}$ of $\mathcal{F}$ such that $\Delta$ is of the form 
\begin{align*}
\overline{\Delta}-\bigcup_{F\in\mathcal{G}}F. 
\end{align*}
For a point $y$ in $\Delta$, we define $\mathfrak{z}_{y}$ a subspace of $\mathfrak{g}$ by 
\begin{align*}
\mathfrak{z}_{y}:=\mathrm{Span}_{\mathbb{R}}\{\, \lambda_{i}\mid y\in H_{\lambda_{i},\kappa_{i}}\, \}. 
\end{align*}
For example, if $y$ is in the interior of $\Delta$ then $\mathfrak{z}_{y}$ is $\{0\}$. 
For a point $p$ in $M$, if we denote the stabilizer at $p$ by $Z_{p}=\{\, t\in T^{m}\mid t\cdot p=p\, \}$, 
then the Lie algebra of $Z_{p}$ coincides with $\mathfrak{z}_{\mu(p)}$. 
Thus, if $\mu(p)$ is in the interior of $\Delta$ then torus action is free at $p$, 
and if $\mu$ maps $p$ to a vertex of $\Delta$ then $p$ is a fixed point. 

Since $(M,J)$ is a toric variety, there exists the intrinsic anti-holomorphic involution $\sigma:M\rightarrow M$ determined by fan $\Sigma$, 
that is, $\sigma^2=id$ and $\sigma_{*}J=-J\sigma_{*}$, where $J$ is the complex structure on $M$. 
This involution satisfies $\sigma(u\cdot p)=\overline{u}\cdot \sigma(p)$, where $u\in T^{m}_{\mathbb{C}}$ acts on $p$. 
Let $M^{\sigma}:=\{\, p\in M\mid \sigma(p)=p\,\}$ be the set of fixed points of $\sigma$, 
that is a submanifold of $M$ with real dimension $m$, we call it the real form of $M$. 

\begin{proposition}
The involution $\sigma:M\rightarrow M$ is anti-symplectic, and consequently $\sigma$ is isometry. 
\end{proposition}
\begin{proof}
Let $U$ be an open dense $T^{m}_{\mathbb{C}}$-orbit. 
For $(w^{1},\dots,w^{m})\in U\cong (\mathbb{C}^{\times})^m$, 
we take the logarithmic holomorphic coordinates $(z^{1},\dots,z^{m})$ with $e^{z^i}=w^i$. 
Since $\omega$ is $T^m$-invariant and the action of $T^m$ is Hamiltonian, 
there exists a function $F\in C^{\infty}(\mathbb{R}^m)$ with the property
\begin{align}
\omega=\frac{\sqrt{-1}}{2}\sum_{i,j=1}^{m} \frac{\partial^{2} F}{\partial x^{i} \partial x^{j}}dz^{i}\wedge d\overline{z}^{j}\quad\mathrm{on}\, \,U,
\end{align}
where $z^{i}=x^{i}+\sqrt{-1}y^{i}$. 
(See Theorem 3.3 in Appendix 2 of \cite{Guillemin2}.)
On $U$, the involution $\sigma$ coincides with the standard complex conjugate $\sigma(z)=\overline{z}$, 
where $\overline{z}=(\overline{z}^{1},\dots,\overline{z}^{m})$. 
Since $\omega$ is $T^{m}$-invariant, note that $F$ is independent of the coordinates $(y^{1},\dots,y^{m})$. 
Thus we have $\sigma^{*}\omega=-\omega$ on $U$. 
Since $U$ is open and dense in $M$, thus we have $\sigma^{*}\omega=-\omega$ on $M$. 
\end{proof}
\section{Lagrangian submanifold}\label{Lag}
Let $n$ be an integer with $0 \leq n \leq m$. 
Take a set of $n$ vectors $\zeta=\{\zeta_{i}\}_{i=1}^{n}\subset\mathfrak{g}$ and a set of $n$ constants $c=\{c_{i}\}_{i=1}^{n}\subset\mathbb{R}$. 
If $n=0$, we take no vectors and no constants. 
We assume that $\{\zeta_{i}\}_{i=1}^{n}$ is linearly independent. 
Then the intersection of $n$ affine hyperplanes $H_{\zeta_{i},c_{i}}$ defines a $(m-n)$-dimensional affine plane. 
We assume that this affine plane intersects in the interior of $\Delta$, 
and we define $\Delta_{\zeta,c}$ a subset of $\Delta$ by
\begin{align*}
\Delta_{\zeta,c}:&=\Delta\cap\biggl( \bigcap_{i=1}^{n}H_{ \zeta_{i}, c_{i} } \biggr)\\
& =\{\,  y \in \Delta \mid \langle y, \zeta_{i} \rangle =c_{i}, \, (i=1,\dots,n)    \,\}. 
\end{align*}

\begin{definition}
Let $V_{\zeta}:=\mathrm{Span}_{\mathbb{R}}\{\zeta_{1},\dots,\zeta_{n}\} \subset \mathfrak{g}$. 
We call a point $y$ in $\Delta$ a $\zeta$-singular point if and only if $V_{\zeta}\cap\mathfrak{z}_{y}\neq\{0\}$, 
and if $V_{\zeta}\cap\mathfrak{z}_{y}=\{0\}$ we call $y$ a $\zeta$-regular point.  
We denote the set of all $\zeta$-singular points and all $\zeta$-regular points in $\Delta$ by $\Delta_{\zeta sing}$ and $\Delta_{\zeta reg}$ respectively. 
Note that $\Delta_{\zeta reg}$ is open dense in $\Delta$. 
\end{definition}
For a point $p$ in $M$, a vector $v$ in $\mathfrak{g}$ generates a tangent vector at $p$ denoted by
$$v_{p}=\frac{d}{dt}\biggl|_{t=0}\exp(tv)\cdot p. $$
This map $\mathfrak{g}\rightarrow T_{p}M$ is a homomorphism. 
Then it is clear that $y$ is a $\zeta$-regular point if and only if the restricted homomorphism $V_{\zeta}\rightarrow T_{p}M$ is injective for a $p$ in $\mu^{-1}(y)$. 
For example, vertices of $\Delta$ are always $\zeta$-singular points and interior points are always $\zeta$-regular points. 

\begin{definition}
We call a point $p$ in $M^{\sigma}$ a $\zeta$-singular point if and only if $\mu(p)$ is a $\zeta$-singular point, 
and if not, we call $p$ a $\zeta$-regular point.  
We denote the set of all $\zeta$-singular points and all $\zeta$-regular points in $M^{\sigma}$ by $M^{\sigma}_{\zeta sing}$ and $M^{\sigma}_{\zeta reg}$ respectively. 
\end{definition}
Note that $M^{\sigma}_{\zeta reg}$ is open dense in $M^{\sigma}$. 

\begin{definition}
We denote the restriction of the moment map on the real form by $\mu^{\sigma}:M^{\sigma}\rightarrow \mathbb{R}^m$. 
We define a subset of $M^{\sigma}$ as the pull-back of $\Delta_{\zeta,c}$ by $\mu^{\sigma}$ by
\begin{align*}
M^{\sigma}_{\zeta,c}:& =(\mu^{\sigma})^{-1}(\Delta_{\zeta,c})\\
& =\{\, p\in M^{\sigma} \mid \langle\mu(p), \zeta_{i}\rangle=c_{i},\, i=1\dots,n \,\}. 
\end{align*}
\end{definition}

\begin{proposition}
If $\Delta_{\zeta,c}$ is contained in $\Delta_{\zeta reg}$, 
then $M^{\sigma}_{\zeta,c}$ is a smooth submanifold of $M^{\sigma}$ 
with $\dim_{\mathbb{R}}M^{\sigma}_{\zeta,c}=m-n$.  
\end{proposition}
\begin{proof}
We define $n$ functions $f_{i}$ ($i=1,\dots,n$) on $M^{\sigma}$ by 
$$f_{i}(p):=\langle \mu(p),\zeta_{i}\rangle -c_{i}.$$ 
Then $M^{\sigma}_{\zeta,c}=\{\, p\in M^{\sigma} \mid f_{i}(p)=0,\, i=1,\dots,n \,\}$. 
By a property of the moment map, for all $p$ in $M^{\sigma}_{\zeta,c}$ we have 
\begin{align*}
df_{i}(p)=d\langle \mu,\zeta_{i}\rangle(p)=-\omega(\zeta_{i,p},\cdot). 
\end{align*}
Since every point in $\Delta_{\zeta,c}$ is $\zeta$-regular, 
the restricted homomorphism $V_{\zeta}\rightarrow T_{p}M$ is injective for all $p$ in $M^{\sigma}_{\zeta,c}$. 
Thus $\{df_{i}\}_{i=1}^{n}$ are linearly independent $1$-forms on $M^{\sigma}_{\zeta,c}$. 
This means that $M^{\sigma}_{\zeta,c}$ is a smooth submanifold of $M^{\sigma}$ by the implicit function theorem. 
\end{proof}
In this section we assume that $\Delta_{\zeta,c}$ is contained in $\Delta_{\zeta reg}$. 
Then $M^{\sigma}_{\zeta,c}$ is a smooth submanifold of $M^{\sigma}$. 
Let $\exp:\mathfrak{g}\rightarrow T^{m}$ be the exponential map. 
Let $\mathbb{Z}_{\mathfrak{g}}(\cong\mathbb{Z}^m)$ be a integral lattice of $\mathfrak{g}$, 
that is a kernel of $\exp:\mathfrak{g}\rightarrow T^{m}$ and $\mathfrak{g}/\mathbb{Z}_{\mathfrak{g}}\cong T^{m}$. 
Let $\frac{1}{2}\mathbb{Z}_{\mathfrak{g}}$ be the set of all elements 
$y$ in $\mathfrak{g}$ such that $2y$ is in $\mathbb{Z}_{\mathfrak{g}}$. 
Then $\frac{1}{2}\mathbb{Z}_{\mathfrak{g}}/\mathbb{Z}_{\mathfrak{g}}\cong\{1,-1\}^{m}$ is a subgroup of $T^{m}$ 
considered as all elements $t$ in $T^{m}$ such that $t^2=e$ identity element. 
Let $V_{\zeta}=\mathrm{Span}_{\mathbb{R}}\{\zeta_{1},\dots,\zeta_{n}\}\subset \mathfrak{g}$. 
Now we construct a manifold $L_{\zeta,c}$ with real dimension $m$. 
 \\
 
{\bf(I) Generic case.} 
For a generic case, let $U$ be an open small ball in $V_{\zeta}$ centered at $0$ such that $U$ and $\frac{1}{2}\mathbb{Z}_{\mathfrak{g}}$ intersect only at $0$. 
Then we define an $m$-dimensional manifold $L_{\zeta,c}$ and a map $F_{\zeta,c}:L_{\zeta,c}\rightarrow M$ by 
$$L_{\zeta,c}=M^{\sigma}_{\zeta,c}\times U\quad and \quad F_{\zeta,c}(p,v):=\exp(v)\cdot p, $$
for $p$ in $M^{\sigma}_{\zeta,c}$ and $v$ in $U$. 
Then $F_{\zeta,c}$ is injective and its image is
\begin{align}\label{L1}
L'_{\zeta,c}:=\{\, \exp(v)\cdot p \mid v \in U,\, p\in M^{\sigma}, \langle \mu(p), \zeta_{j}\rangle=c_{j},\, j=1,\dots,n \,\}. 
\end{align}
 \\
 
\noindent
{\bf(II) Special case.} 
However if the set of vectors $\zeta=\{\zeta_{i}\}_{i=1}^{n}$ satisfies the following special condition we can take $L_{\zeta,c}$ as explained below. 
\begin{definition}
We say that $\zeta$ satisfies the special condition if there exists a set of $n$ vectors $v=\{v_{j}\}_{j=1}^{n}$ in $V_{\zeta}\cap\mathbb{Z}_{\mathfrak{g}}$ such that 
$v$ is a base of $V_{\zeta}$ and $v$ is a generator of $V_{\zeta}\cap\mathbb{Z}_{\mathfrak{g}}$ over $\mathbb{Z}$. 
\end{definition}

If $\zeta$ satisfies the special condition, we replace $U$ in case (I) by $T_{\zeta}:=V_{\zeta}/(V_{\zeta}\cap\mathbb{Z}_{\mathfrak{g}})$ and 
we define an $m$-dimensional manifold $L_{\zeta,c}$ and a map $F_{\zeta,c}:L_{\zeta,c}\rightarrow M$ by 
$$L_{\zeta,c}=M^{\sigma}_{\zeta,c}\times T_{\zeta}\quad and \quad F_{\zeta,c}(p,[v]):=\exp(v)\cdot p, $$
for $p$ in $M^{\sigma}_{\zeta,c}$ and $[v]$ in $T_{\zeta}=V_{\zeta}/(V_{\zeta}\cap\mathbb{Z}_{\mathfrak{g}})$, this map is well defined. 
Since $T_{\zeta}\cong T^{n}$ which is a subtorus of $T^{m}$, $L_{\zeta,c}$ is diffeomorphic to $M^{\sigma}_{\zeta,c}\times T^{m}$. 
We denote the subgroup $(V_{\zeta}\cap\frac{1}{2}\mathbb{Z}_{\mathfrak{g}})/(V_{\zeta}\cap\mathbb{Z}_{\mathfrak{g}})$ of $T_{\zeta}$ by $K_{\zeta}$. 
Then of course $K_{\zeta}$ acts on $T_{\zeta}$ freely and $K_{\zeta}$ also acts on $M^{\sigma}_{\zeta,c}$ as 
$$[k]\cdot p:=\exp(k)\cdot p$$
for $[k]$ in $K_{\zeta}$ and $p$ in $M^{\sigma}_{\zeta,c}$. 
Thus $K_{\zeta}$ acts on $L_{\zeta,c}=M^{\sigma}_{\zeta,c}\times T_{\zeta}$ as a diagonal action and this action is free. 
Hence we have an $m$-dimensional manifold $\tilde{L}_{\zeta,c}$ by 
$$\tilde{L}_{\zeta,c}:=(M^{\sigma}_{\zeta,c}\times T_{\zeta})/K_{\zeta}. $$
In general $F_{\zeta, c}:L_{\zeta,c}\rightarrow M$ is not injective and one can show that 
$F_{\zeta,c}(p_{1},[v_{1}])=F_{\zeta,c}(p_{2},[v_{2}])$ if and only if there exists a $[k]$ in $K_{\zeta}$ such that $[k]\cdot(p_{1},[v_{1}])=(p_{2},[v_{2}])$. 
Thus the image of $F_{\zeta,c}$ written by 
\begin{align}\label{L2}
L'_{\zeta,c}:=\{\, \exp(v)\cdot p \mid v \in V_{\zeta},\, p\in M^{\sigma}, \langle \mu(p), \zeta_{j}\rangle=c_{j},\, j=1,\dots,n \,\}
\end{align} 
is diffeomorphic to $\tilde{L}_{\zeta,c}$. 
Note that $\tilde{L}_{\zeta,c}$ is a $T^{n}$-bundle over a smooth $(m-n)$-dimensional manifold  $M^{\sigma}_{\zeta,c}/K_{\zeta}$. 

\begin{remark}
If we take no vectors $\zeta$ and no constants $c$, that is $\zeta=\emptyset$ and $c=\emptyset$, 
then $L_{\zeta,c}$ becomes the real form $M^{\sigma}$ itself hence $L_{\zeta,c}$ has no torus factors. 
On the other hand, if the number of vectors in $\zeta$ is max, that is $m$, then $M^{\sigma}_{\zeta,c}=\{pt\}$ thus 
$L_{\zeta,c}$ is diffeomorphic to $T^{m}$. 
Hence roughly speaking, the number of vectors in $\zeta$ is the dimension of torus factors in $L_{\zeta,c}$. 
\end{remark}
@\\
 
From now we consider both cases (I) and (II) above. 
\begin{theorem}\label{lag}
$F_{\zeta,c}:L_{\zeta,c}\rightarrow M$ is a Lagrangian submanifold. 
\end{theorem}
\begin{proof}
In this proof, we write $F_{\zeta,c}$ by $F$ for short. 
It is clear that we only have to prove in the case (I). 
First we prove that $F$ is an immersion map. 
Fix a point $x=(p,v)$ in $L_{\zeta,c}=M^{\sigma}_{\zeta,c}\times U$. 
Then we have a decomposition 
$$T_{x}L_{\zeta,c}=T_{p}M^{\sigma}_{\zeta,c}\oplus T_{v}U, $$
and note that $T_{v}U\cong V_{\zeta}$. 
Take tangent vectors $X,X_{1},X_{2}$ in $T_{p}M^{\sigma}_{\zeta,c}$. We have
$$F_{*}X=t_{v*}X, $$
where we put $t_{v}:=\exp(v)$ for short and 
we identify an element $t_{v}$ in $T^m$ with a left transition map $t_{v}:M\rightarrow M$. 
Take vectors $Y,Y_{1},Y_{2}$ in $T_{v}U\cong V_{\zeta}$. We have
$$F_{*}Y=t_{v*}Y_{p}. $$
Note that $\sigma_{*}X=X$ since $X$ is tangent to the real form 
and $\sigma_{*}Y_{p}=-Y_{p}$ since $\sigma(u \cdot p)=u^{-1}\cdot p$ for all $u$ in $T^m$. 
Thus we have 
$$g(F_{*}X,F_{*}Y)=g(X,Y_{p})=(\sigma^{*}g)(X,Y_{p})=-g(X,Y_{p}), $$
and this means that $g(F_{*}X,F_{*}Y)=0$ and 
$F_{*}(T_{p}M^{\sigma}_{\zeta,c})$ and $F_{*}(T_{v}U)$ are orthogonal to each other. 
It is clear that $F_{*}$ restricted on $T_{p}M^{\sigma}_{\zeta,c}$ is injective 
and $F_{*}$ restricted on $T_{v}U$ is also injective. 
Thus this means that $F_{*}$ is injective on $T_{x}L_{\zeta,c}$ and $F$ is an immersion map. 

Next we prove that $F$ is a Lagrangian, that is $F^{*}\omega=0$. 
It is easy to see $(F^{*}\omega)(X_{1},X_{2})=0$ and $(F^{*}\omega)(Y_{1},Y_{2})=0$. 
We can also prove that $(F^{*}\omega)(X,Y)=0$. 
Actually we have
\begin{align*}
(F^{*}\omega)(X,Y)=\omega(X,Y_{p})=X(\langle \mu,Y\rangle )=0,
\end{align*}
since we can write $Y=a^{1}\zeta_{1}+\cdots+a^{n}\zeta_{n}$ for some coefficients $a^{k}$ 
then $\langle \mu,Y\rangle$ is a constant $a^{1}c_{1}+\dots+a^{n}c_{n}$ on $M^{\sigma}_{\zeta,c}$. 
\end{proof}
\section{Lagrangian angle}\label{angle}
In above sections, the ambient space $(M,\omega,g,J)$ is a toric K\"ahler manifold. 
From this section, we assume that the canonical line bundle $K_{M}$ of $(M,J)$ is trivial. 
This condition is equivalent to that there exists a vector $\gamma$ in $\mathbb{Z}_{g}^{*}$ such that 
$\langle \gamma, \lambda_{i} \rangle=1$ for all $i=1,\dots,d$, 
where $\lambda_{i}$ is a primitive generator of a 1-dimensional cone of fan $\Sigma$ of $M$, see Section \ref{TKm}. 
In fact, if such a vector $\gamma=(\gamma_{1},\dots,\gamma_{m})$ exists, a holomorphic $(m,0)$-form 
\begin{align}\label{Omega}
\Omega_{\gamma}:=e^{\gamma_{1}z^{1}+\dots+\gamma_{m}z^{m}}dz^{1}\wedge\dots\wedge dz^{m}
\end{align}
written by logarithmic holomorphic coordinates on an open dense $(\mathbb{C}^{*})^m$-orbit 
can be extend over $M$ as a nowhere vanishing holomorphic $(m,0)$-form. 
We call this $(M, \omega, g, J, \Omega_{\gamma})$ a toric almost Calabi--Yau manifold. 

In general an $m$-dimensional K\"ahler manifold $(M, \omega, g, J)$ with nowhere vanishing holomorphic $(m,0)$-form $\Omega$ 
is called an almost Calabi--Yau manifold, and for a Lagrangian immersion $F:L\rightarrow M$ 
we can define the Lagrangian angle $\theta_{F}:L\rightarrow \mathbb{R}/\pi\mathbb{Z}$ as follows. 
For $x$ in $L$, take a local chart $(U, (x^{1},\dots,x^{m}))$ around $x$, then $F^{*}\Omega$ is a 
$\mathbb{C}^{*}$-valued $m$-form on $U$, so there exists a $\mathbb{C}^{*}$-valued function $h_{U}$ on $U$ such that 
$$F^{*}\Omega=h_{U}(x^{1},\dots,x^{m})dx^{1}\wedge\dots\wedge dx^{m}$$
on $U$, and we define the Lagrangian angle $\theta_{F}:L\rightarrow \mathbb{R}/\pi\mathbb{Z}$ by 
$$\theta_{F}(x):=\arg(h_{U}(x)) \mod{\pi}.$$ 
This definition is independent of the choice of local charts. 
It is clear that if $L$ is oriented we can lift $\theta_{F}$ to a $\mathbb{R}/2\pi\mathbb{Z}$-valued function $\theta_{F}:L\rightarrow \mathbb{R}/2\pi\mathbb{Z}$. 
If we can lift $\theta_{F}$ to a $\mathbb{R}$-valued function $\theta_{F}:L\rightarrow \mathbb{R}$ then $F:L\rightarrow M$ is called Maslov zero, 
and furthermore if $\theta_{F}$ is constant $\theta_{0}$ then $F:L\rightarrow M$ is called a special Lagrangian submanifold with phase $e^{i\theta_{0}}$. 

In \cite{Behrndt}, Behrndt introduced the notion of the generalized mean curvature vector field $K$ 
for a Lagrangian immersion $F:L\rightarrow M$ in an almost Calabi--Yau manifold. 
The generalized mean curvature vector field $K$ is defined by 
\begin{align}\label{defK}
K:=H-m\nabla \psi^{\bot}
\end{align}
where $H$ is the mean curvature vector field of the immersion $F:L\rightarrow (M,g)$, $\psi$ is a function on $M$ defined by the following equation; 
\begin{align}\label{psi}
e^{2m\psi}\frac{\omega^m}{m!}=(-1)^{\frac{m(m-1)}{2}}\biggl(\frac{i}{2}\biggr)^m\Omega\wedge\overline{\Omega}, 
\end{align}
and $\nabla \psi^{\bot}$ is the normal part of the gradient of $\psi$. 
By the definition of $K$, if $M$ is a Calabi--Yau manifold, that is $\psi\equiv 0$, then the generalized mean curvature vector field $K$ 
coincides with the mean curvature vector field $H$. 
In Proposition 4.8 in \cite{Behrndt2}, Behrndt proved the relation between $K$ and $\theta_{F}$ which is written by 
\begin{align}\label{K}
K=J\nabla\theta_{F}. 
\end{align}
Thus $K\equiv 0$ is equivalent to that $L$ is a special Lagrangian submanifold. 

Furthermore, in this paper, we introduce the notion of \textit{weighted Hamiltonian stationary} for a Lagrangian immersion $F:L\rightarrow M$ into 
an almost Calabi--Yau manifold  $(M, \omega, g, J, \Omega)$ with $\psi$ defined by (\ref{psi}).
\begin{definition}
Let $\theta_{F}$ be the Lagrangian angle of $F:L\rightarrow M$. 
If $\Delta_{f}\theta_{F}=0$ then we call $F:L\rightarrow M$ a weighted Hamiltonian stationary Lagrangian submanifold. 
\end{definition}
Here $f$ is a function on $L$ defined by $f:=-mF^{*}\psi$ and 
$\Delta_{f}$ is the weighted Laplacian on Riemannian manifold $(L,F^{*}g)$. 
In general, for a Riemannian manifold $(N, h)$ with a function $f$, the weighted Laplacian with respect to $f$ is defined by
$\Delta_{f}u:=\Delta u +\langle\nabla u,\nabla f \rangle$. 
Thus if $M$ is a Calabi--Yau manifold, that is $\psi=0$, then the notion of weighted Hamiltonian stationary is equivalent to the Hamiltonian stationary condition i.e. $\Delta\theta_{F}=0$. 
For the meaning of the weighted Hamiltonian stationary condition, See Appendix \ref{app}. 
Note that $\Delta_{f}$ is the standard Laplace operator on $L$ with respect to a Riemannian metric $F^{*}(e^{2\psi}g)$. 

In this section, we compute the Lagrangian angle of the concrete example $F_{\zeta,c}:L_{\zeta,c}\rightarrow M$ constructed in Section \ref{Lag}, 
and show some properties of $F_{\zeta,c}:L_{\zeta,c}\rightarrow M$. 

Let $(M, \omega, g, J, \Omega_{\gamma})$ be an $m$-dimensional  toric almost Calabi--Yau manifold and 
$F_{\zeta,c}:L_{\zeta,c}\rightarrow M$ be a Lagrangian immersion constructed by $\zeta=\{\zeta_{1},\dots,\zeta_{n}\}\subset\mathfrak{g}$ 
and $c=\{c_{1},\dots,c_{n}\}\subset\mathbb{R}$ explained in Section \ref{Lag}. 
\begin{theorem}\label{lagang}
Let $\theta$ be the Lagrangian angle of $F_{\zeta,c}:L_{\zeta,c}\rightarrow M$ then we have
$$\theta(x)=2\pi\langle\gamma,v\rangle + \frac{\pi}{2}n \mod{\pi}$$
for $x=(p,v)$ in $L_{\zeta,c}=M^{\sigma}_{\zeta,c}\times U$ in the case {\rm(I)} and for $x=(p,[v])$ in $L_{\zeta,c}=M^{\sigma}_{\zeta,c}\times T_{\zeta}$ in the case {\rm(II)}. 
\end{theorem}
\begin{proof}
In this proof, we write $F_{\zeta,c}$ by $F$ for short. 
It is clear that we only have to prove in the case (I). 
Let $M^{\sigma}$ be a real form of $M$ and $\mathfrak{g}$ be a Lie algebra of $T^m$. 
We define a map $\tilde{F}:M^{\sigma}\times\mathfrak{g} \rightarrow M$ by 
$$\tilde{F}(p,v):=\exp(v)\cdot p. $$
Remember that $L_{\zeta,c}=M^{\sigma}_{\zeta,c}\times U$, and $M^{\sigma}_{\zeta,c}$ is an $(m-n)$-dimensional submanifold in $M^{\sigma}$ and $U$ is an $n$-dimensional submanifold in $\mathfrak{g}$. 
Thus we write the inclusion map $L_{\zeta,c}$ into $M^{\sigma}\times\mathfrak{g}$ by 
$$\iota=(\iota_{1},\iota_{2}):L_{\zeta,c}=M^{\sigma}_{\zeta,c}\times U \hookrightarrow M^{\sigma}\times \mathfrak{g}. $$
Then the map $F:L_{\zeta,c}\rightarrow M$ coincides with $\tilde{F}\circ\iota$ by the definition of $F$, 
so we compute $\iota^{*}(\tilde{F}^{*}\Omega_{\gamma})$ to compute $F^{*}\Omega_{\gamma}$. 
It is enough to prove this theorem on an open dense $(\mathbb{C}^{*})^{m}$-orbit,  
so we take a logarithmic holomorphic coordinates $(z^{1},\dots,z^{m})$, 
then $(x^{1},\dots,x^{m})$ define local coordinates on the real form $M^{\sigma}$, where $z^{j}=x^{j}+iy^{j}$. 
Let $(t^{1},\dots,t^{m})$ be coordinates of $\mathfrak{g}\cong\mathbb{R}^{m}$ 
then we have a local expression of a map $\tilde{F}:M^{\sigma}\times\mathfrak{g} \rightarrow M$ by 
$$\tilde{F}(x^{1},\dots,x^{m},t^{1},\dots,t^{m})=(x^{1}+2\pi i t^{1},\dots,x^{m}+2\pi i t^{m}). $$
Since $\Omega_{\gamma}=e^{\gamma_{1}z^{1}+\dots+\gamma_{m}z^{m}}dz^{1}\wedge\dots\wedge dz^{m}$, we have
$$\tilde{F}^{*}\Omega_{\gamma}=e^{(\gamma_{1}x^{1}+\dots+\gamma_{m}x^{m})+2\pi i(\gamma_{1}t^{1}+\dots+\gamma_{m}t^{m})}(dx^{1}+2\pi i dt^{1})\wedge\dots\wedge(dx^{m}+2\pi i dt^{m}).$$
Since $L_{\zeta,c}=M^{\sigma}_{\zeta,c}\times U$, and $M^{\sigma}_{\zeta,c}$ is an $(m-n)$-dimensional submanifold in $M^{\sigma}$ and $U$ is an $n$-dimensional submanifold in $\mathfrak{g}$, 
in the expansion of $(dx^{1}+2\pi i dt^{1})\wedge\dots\wedge(dx^{m}+2\pi i dt^{m})$, only differential forms such as 
$$(2\pi i)^{n} dx^{I}\wedge dt^{J}$$
with $\sharp I =m-n$ and $\sharp J= n$ do not vanish after pull-back by $\iota$. 
Thus the argument of $F^{*}\Omega_{\gamma}=\iota^{*}(\tilde{F}^{*}\Omega_{\gamma})$ is the argument of 
$$(2\pi i)^{n}e^{\langle\gamma,p\rangle+2\pi i\langle\gamma,v\rangle}, $$
that is $2\pi\langle\gamma,v\rangle + \frac{\pi}{2}n \mod{\pi}$. 
\end{proof}
Then the following corollary is clear. 
\begin{corollary}
$F_{\zeta,c}:L_{\zeta,c}\rightarrow M$ is a special Lagrangian submanifold if and only if 
$\langle\gamma,\zeta_{i}\rangle=0$ for all $i=1,\dots,n$. 
\end{corollary}
It is clear that the real form $M^{\sigma}$, that is the case of $n=0$, is always a special Lagrangian submanifold and 
every torus fiber, that is the case of $n=m$, is not a special Lagrangian submanifold. 
\begin{theorem}
$F_{\zeta,c}:L_{\zeta,c}\rightarrow M$ is weighted Hamiltonian stationary. 
\end{theorem}
\begin{proof}
In this proof, we write $F_{\zeta,c}$ by $F$ for short. 
We only have to prove that $\Delta_{f}\theta=0$ in the case (I) that $L_{\zeta,c}=M^{\sigma}_{\zeta,c}\times U$. 
As noted above, $\Delta_{f}$ is the standard Laplace operator on $L$ with respect to a Riemannian metric $F^{*}(e^{2\psi}g)$. 
Since $g$ is invariant under the torus action and it is easily seen that $\psi$ is also torus invariant by the equation (\ref{Omega}) and (\ref{psi}), 
so the metric $e^{2\psi}g$ is also a torus invariant metric on $M$. 
Since $F:L_{\zeta,c}\rightarrow M$ is given by $F(p,v):=\exp(v)\cdot p$ and $e^{2\psi}g$ is a torus invariant metric on $M$, 
the metric $F^{*}(e^{2\psi}g)$ on $L$ is independent of the $U$-factor of $L_{\zeta,c}$. 
Furthermore in the proof of Theorem \ref{lag} we prove that $F_{*}(TM^{\sigma}_{\zeta,c})$ and $F_{*}(TU)$ are orthogonal, 
thus $F^{*}(e^{2\psi}g)$ is a product metric over $M^{\sigma}_{\zeta,c}$ and $U$ locally. 
By Theorem \ref{lagang}, the Lagrangian angle is given by $\theta(p,v)=2\pi\langle\gamma,v\rangle + \frac{\pi}{2}n$, 
it is independent of $M^{\sigma}_{\zeta,c}$-factor of $L_{\zeta,c}$ and affine on $U$-factor. 
Then one can easily prove that $\Delta_{f}\theta=0$. 
\end{proof}
\section{Mean curvature flow}\label{MCF}
In this section, we consider generalized Lagrangian mean curvature flows. 
In general, a generalized Lagrangian mean curvature flow is defined in an almost Calabi--Yau manifold $(M,\omega, g, J, \Omega)$. 
Let $F_{0}:L\rightarrow M$ be a Lagrangian immersion, 
then a one parameter family of Lagrangian submanifolds $F:L\times I \rightarrow M$ is called a solution of a generalized Lagrangian mean curvature flow 
with initial condition $F_{0}$ if it moves along its generalized Lagrangian mean curvature vector field $K$ defined in (\ref{defK}), that is, 
\begin{align}\label{gmcf}
\biggl(\frac{\partial F}{\partial t}\bigg)^{\bot}=K_{t}\hspace{5mm}and\hspace{5mm}F(\cdot,0)=F_{0},
\end{align}
where $K_{t}$ is the generalized Lagrangian mean curvature vector field of immersion $F_{t}:L\rightarrow M$ defined by $F_{t}(p):=F(p,t)$. 
Of course, if $M$ is a Calabi--Yau manifold then a generalized Lagrangian mean curvature flow is an ordinary Lagrangian mean curvature flow. 
It is clear that on a special Lagrangian submanifold $K$=0 by the equation (\ref{K}) 
thus a special Lagrangian submanifold is a stationary solution of a generalized Lagrangian mean curvature flow. 
In general a generalized Lagrangian mean curvature flow develops some singularities in a finite time, so here we define a notion of 
a generalized Lagrangian mean curvature flow with some singularities and topological changes. 
\begin{definition}
Let $(M,\omega, g, J, \Omega)$ be a real $2m$-dimensional almost Calabi--Yau manifold and $\{L_{t}\}_{t\in I}$ be a one parameter family of subsets in $M$. 
Then we call $\{L_{t}\}_{t\in I}$ a solution of a generalized Lagrangian mean curvature flow with singularities and topological changes if 
there exists a real $m$-dimensional manifold $L$ and a solution of a generalized Lagrangian mean curvature flow $F:L\times I\rightarrow M$ such that 
$F_{t}:L\rightarrow M$ is an embedding into $L_{t}$ and $m$-dimensional Hausdorff measure of $L_{t}\setminus F_{t}(L)$ is zero, i.e.  
\begin{align}\label{gmcf2}
F_{t}(L)\subset L_{t}\hspace{5mm}and\hspace{5mm}\mathcal{H}^{m}(L_{t}\setminus F_{t}(L))=0. 
\end{align}
\end{definition}
It means that $\{L_{t}\}_{t\in I}$ is almost parametrized by a smooth solution of a generalized Lagrangian mean curvature flow. 

The purpose of this section is to observe how our concrete examples $F_{\zeta,c}:L_{\zeta,c}\rightarrow M$ move along the generalized Lagrangian mean curvature flow. 
Let $(M,\omega, g, J, \Omega_{\gamma})$ be a toric almost Calabi--Yau manifold and 
$F_{\zeta,c}:L_{\zeta,c}\rightarrow M$ be a Lagrangian submanifold constructed in Section \ref{Lag} by data $\zeta=\{\zeta_{1},\dots,\zeta_{n}\}\subset\mathfrak{g}$ and $c=\{c_{1},\dots,c_{n}\}\subset\mathbb{R}$. 
Let $$c_{i}(t):=c_{i}-2\pi\langle\gamma,\zeta_{i}\rangle t$$ for $t\in\mathbb{R}$ and we denote $c(t):=\{c_{1}(t),\dots,c_{n}(t)\}$. 
We define an open interval $I$ by 
$$I:=\biggl\{\,t\in\mathbb{R}\,\bigg|\,\mathrm{Int}\Delta\cap\biggl( \bigcap_{i=1}^{n}H_{ \zeta_{i}, c_{i}(t) } \biggr)\neq\emptyset\,\biggr\}, $$
by the assumption of $\zeta$ and $c$ we have $0\in I$. 
\begin{theorem}
A one parameter family of subsets $\{L'_{\zeta,c(t)}\}_{t\in I}$ defined by (\ref{L1}) in the case (I) or by (\ref{L2}) in the case (II) is 
a solution of a generalized Lagrangian mean curvature flow with singularities and topological changes. 
\end{theorem}
\begin{proof}
It is sufficient to prove this theorem in the case (I). 
First we define 
$$\Delta''_{\zeta,c(t)}:=\mathrm{Int}\Delta\cap\biggl( \bigcap_{i=1}^{n}H_{ \zeta_{i}, c_{i}(t) }\biggr). $$
This is an $(m-n)$-dimensional affine submanifold in $\Delta$. 
In fact, all $\Delta''_{\zeta,c(t)}$ are diffeomorphic to each other, and 
$(m-n)$-dimensional Hausdorff measure of $\Delta_{\zeta,c(t)}\setminus\Delta''_{\zeta,c(t)}$ is zero, 
since it is contained in $\partial\Delta$. 
Next we define 
$$M''^{\sigma}_{\zeta,c(t)}:=(\mu^{\sigma})^{-1}(\Delta''_{\zeta,c(t)})\hspace{5mm}and\hspace{5mm}
L''_{\zeta,c(t)}:=M''^{\sigma}_{\zeta,c(t)} \times U.$$
Then $M''^{\sigma}_{\zeta,c(t)}$ is an $(m-n)$-dimensional submanifold in $M$ and $L''_{\zeta,c(t)}$ is an $m$-dimensional manifold contained in $L_{\zeta,c(t)}$. 
As same as $\Delta''_{\zeta,c(t)}$, all $M''^{\sigma}_{\zeta,c(t)}$ are diffeomorphic to each other, and 
$(m-n)$-dimensional Hausdorff measure of $M^{\sigma}_{\zeta,c(t)}\setminus M''^{\sigma}_{\zeta,c(t)}$ is zero and 
$m$-dimensional Hausdorff measure of $L_{\zeta,c(t)}\setminus L''_{\zeta,c(t)}$ is also zero. 
Thus we can take a one parameter family of diffeomorphisms 
$$G_{t}:M''^{\sigma}_{\zeta,c}\rightarrow M''^{\sigma}_{\zeta,c(t)}, $$
for all $t\in I$ and $G_{t}$ induces a one parameter family of diffeomorphisms 
$$\tilde{G}_{t}:L''_{\zeta,c}\rightarrow L''_{\zeta,c(t)}$$
by $\tilde{G}_{t}(p,v):=(G_{t}(p),v)$. 
Then we have a one parameter family of maps $F:L''_{\zeta,c}\times I\rightarrow M$ by 
$$F_{t}(p,v):=F_{\zeta,c(t)}\circ \tilde{G}_{t}(p,v)=\exp(v)\cdot G_{t}(p). $$
It is clear that 
$$F_{t}(L''_{\zeta,c})=F_{\zeta,c(t)}(\tilde{G}_{t}(L''_{\zeta,c}))=F_{\zeta,c(t)}(L''_{\zeta,c(t)})\subset L'_{\zeta,c(t)}, $$
where remember that 
$$L'_{\zeta,c(t)}=\{\, \exp(v)\cdot p \mid v \in U,\, p\in M^{\sigma}, \langle \mu(p), \zeta_{j}\rangle=c_{j}(t),\, j=1,\dots,n \,\}. $$
Since torus action is free on $M''^{\sigma}_{\zeta,c(t)}$, one can easily prove that $F_{t}$ is embedding for all $t$ and 
$m$-dimensional Hausdorff measure of $L'_{\zeta,c(t)}\setminus F_{t}(L''_{\zeta,c})$ is zero. 

Hence the remainder we have to prove is to prove that $F:L''_{\zeta,c}\times I\rightarrow M$ is a solution of a 
generalized Lagrangian mean curvature flow. 
Since both $K_{t}$ and the normal part of $\partial F/\partial t$ are sections of normal bundle and 
$F_{t}:L''_{\zeta,c}\rightarrow M$ is a Lagrangian submanifold, 
it is enough to prove 
\begin{align}\label{eq1}
\omega(\frac{\partial F}{\partial t},F_{t*}Z)=\omega(K_{t},F_{t*}Z)
\end{align}
for all tangent vector $Z$ on $L''_{\zeta,c}$ to prove the equation (\ref{gmcf}). 
Fix a point $x=(p,v)$ in $L''_{\zeta,c}=M''^{\sigma}_{\zeta,c}\times U$. 
Since we have a decomposition 
$$T_{x}L''_{\zeta,c}=T_{p}M''^{\sigma}_{\zeta,c}\oplus T_{v}U$$ 
and note that $T_{v}U\cong V_{\zeta}$, 
a tangent vector $Z$ is written by $Z=X+Y$ for some tangent vectors $X$ in $T_{p}M''^{\sigma}_{\zeta,c}$ and $Y$ in $V_{\zeta}$. 
For $X$ and $Y$, we have 
$$F_{t*}X=\exp(v)_{*}(G_{t*}X)\hspace{5mm}and\hspace{5mm}F_{t*}Y=\exp(v)_{*}(Y_{G_{t}(p)}). $$
For $X$, we have
\begin{align*}
\omega(\frac{\partial F}{\partial t},F_{t*}X)=\omega(\exp(v)_{*}(\frac{\partial G}{\partial t}),\exp(v)_{*}(G_{t*}X))=\omega(\frac{\partial G}{\partial t},G_{t*}X)=0. 
\end{align*}
The second equality follows from the torus invariance of $\omega$, 
and the third equality follows from that both $\partial G/\partial t$ and $G_{t*}X$ are tangent to real form and it is a Lagrangian. 
If we use the equation (\ref{K}), we have 
\begin{align*}
\omega(K_{t},F_{t*}X)=\omega(J\nabla \theta_{F_{t}},F_{t*}X)=-g(\nabla \theta_{F_{t}},F_{t*}X)=-X\theta_{F_{t}}=0, 
\end{align*}
since $\theta_{F_{t}}(p,v)=2\pi\langle\gamma,v\rangle + \frac{\pi}{2}n$ by Theorem \ref{lagang} and it is independent of $M''^{\sigma}_{\zeta,c}$ part. 
Thus the equation (\ref{eq1}) holds for $X$. 
Next for Y, we have 
\begin{align*}
\omega(\frac{\partial F}{\partial t},F_{t*}Y)&=\omega(\frac{\partial G}{\partial t},Y_{G_{t}(p)})=\frac{\partial G}{\partial t}\langle\mu, Y\rangle=\frac{\partial }{\partial t}\langle\mu\circ G_{t}, Y\rangle\\
&=\frac{\partial }{\partial t}\langle\mu\circ G_{t}, a^{1}\zeta_{1}+\dots+a^{n}\zeta_{n}\rangle\\
&=\frac{\partial }{\partial t}(a^{1}c_{1}(t)+\dots+a^{n}c_{n}(t))\\
&=-2\pi\langle \gamma, Y\rangle. 
\end{align*}
The second equality follows from the assumption of the moment map $\mu$. 
In the fourth equality we put $Y=a^{1}\zeta_{1}+\dots+a^{n}\zeta_{n}$ for some coefficients $a^{i}$ and 
the fifth equality follows from the definition of $M''^{\sigma}_{\zeta,c(t)}$. 
In the last equality, remember that $c_{i}(t)$ is defined by  $c_{i}(t):=c_{i}-2\pi\langle\gamma,\zeta_{i}\rangle t$. 
If we use the equation (\ref{K}), we have 
\begin{align*}
\omega(K_{t},F_{t*}Y)=\omega(J\nabla \theta_{F_{t}},F_{t*}Y)=-g(\nabla \theta_{F_{t}},F_{t*}Y)=-Y\theta_{F_{t}}=-2\pi\langle\gamma,Y\rangle. 
\end{align*}
Thus the equation (\ref{eq1}) holds for $Y$ and it is proved that $F:L''_{\zeta,c}\times I\rightarrow M$ is a solution of a generalized Lagrangian mean curvature flow. 
\end{proof}
\section{Examples}
In this section we give some examples of our main theorems. 
First we explain that if the ambient space $M$ is $\mathbb{C}^m$ then our examples coincide with those constructed by Lee and Wang in \cite{LeeWang}. 
 \\
 
\noindent
{\bf Example 5.1.} 
Let $(\mathbb{C}^{m},\omega,g,J,\Omega)$ be a standard complex plane with a holomorphic volume form $\Omega=dw^{1}\wedge\dots\wedge dw^{m}$ by the standard coordinates $w$. 
If we write $w_{i}=e^{z_{i}}$ where $w_{i}\neq 0$ then $\Omega$ is written by $\Omega=e^{z^{1}+\dots+z^{m}}dz^{1}\wedge\dots\wedge dz^{m}$. 
Hence we can take $\gamma$ as $\gamma=(1,\dots,1)$. 
A moment map is given by $\mu(w)=\frac{1}{2}(|w^{1}|^2,\dots,|w^{m}|^2)$ and a moment polytope is given by 
$$\Delta=\{\,y\in\mathbb{R}^{m}\mid\langle y,\lambda_{i}\rangle\geq 0,\, i=1,\dots,m\,\}$$
where $\lambda_{i}:=e_{i}$ the $i$-th standard base and then we have $\langle \gamma,\lambda_{i}\rangle=1$ for all $i$. 
The real form of $\mathbb{C}^{m}$ is $\mathbb{R}^{m}$ and of course $\mathbb{R}^{m}$ can be constructed by gluing from $2^{m}$-copies of $\Delta$. 
Take $\zeta=(\zeta_{1},\dots,\zeta_{m})\in\mathbb{R}^{m}$ satisfying $\langle \gamma,\zeta\rangle>0$ and $c=0$. 
Since 
$$c(t)=c-2\pi\langle \gamma,\zeta\rangle t=-2\pi t\langle \gamma,\zeta\rangle=-2\pi t\sum_{j=1}^{m}\zeta_{j}$$
and $\Delta_{\zeta,c(t)}=\{\,  y \in \Delta \mid \langle y, \zeta \rangle =c(t)\,\}$, we have 
\begin{align*}
M^{\sigma}_{\zeta,c(t)}&=(\mu|_{\mathbb{R}^{m}})^{-1}(\Delta_{\zeta,c(t)})\\
&=\biggl\{\,  x \in \mathbb{R}^{m} \,\bigg|\, \sum_{j=1}^{m}\zeta_{j}x^{2}_{j}=-4\pi t\sum_{j=1}^{m}\zeta_{j}\,\biggr\}, 
\end{align*}
and $L'_{\zeta,c(t)}$ the image of $F_{\zeta,c(t)}:L_{\zeta, c}\rightarrow \mathbb{C}^{m}$ is given by 
\begin{align*}
L'_{\zeta,c(t)}=\biggl\{\, (x_{1}e^{2\pi i\zeta_{1}s}&,\dots,x_{m}e^{2\pi i\zeta_{m}s})  \in \mathbb{R}^{m} \,\bigg|\, 0\leq s \leq 1, \\
 &\sum_{j=1}^{m}\zeta_{j}x^{2}_{j}=-4\pi t\sum_{j=1}^{m}\zeta_{j},\, x=(x_{1},\dots,x_{m})\in\mathbb{R}^{m}\,\biggr\}. 
\end{align*}
This $L'_{\zeta,c(t)}$ coincides with $V_{t}$ in Theorem 1.1 in \cite{LeeWang} and Lee and Wang proved that 
$V_{t}$ is Hamiltonian stationary and $\{V_{t}\}_{t\in\mathbb{R}}$ form an eternal solution for Brakke flow. 
Hence our theorems can be considered as some kind of generalization of example of Lee and Wang \cite{LeeWang} to toric almost Calabi--Yau manifolds. 
 \\

\noindent
{\bf Example 5.2.} 
Let $M=K_{\mathbb{P}^2}$ be a canonical line bundle of $\mathbb{P}^2$. 
Then a moment polytope is given by 
$\Delta=\{\,y\in\mathbb{R}^{3}\mid\langle y,\lambda_{i}\rangle\geq \kappa_{i},\, i=1,\dots,4\,\}$ where 
$$\lambda_{1}=(0,0,1),\, \lambda_{2}=(1,0,1),\, \lambda_{3}=(0,1,1),\, \lambda_{4}=(-1,-1,1)$$ 
and $\kappa_{1}=\kappa_{2}=\kappa_{3}=0$, $\kappa_{4}=-1$. 
Of course $M$ is a toric almost Calabi--Yau manifold since we can take $\gamma=(0,0,1)$ so that 
$\langle \gamma,\lambda_{i}\rangle=1$ for all $i$. 
For example, take 
$$\zeta=(3,1,5)\hspace{5mm}\mathrm{and}\hspace{5mm}c=5. $$
Then $\Delta_{\zeta,c(t)}$ is written by 
$$\Delta_{\zeta,c(t)}=\{\,y\in\Delta\mid\langle y,\zeta\rangle=5-10\pi t\,\}$$ 
since $c(t)=c-2\pi\langle\gamma,\zeta\rangle t$ and $t\geq 0$. 
We write each facet of $\Delta$ by $F_{i}:=\{\,y\in\Delta\mid \langle y,\lambda_{i}\rangle=\kappa_{i}\,\}$ for $i=1,2,3,4$. 

By simple calculation, one can easily see that 
when $0\leq t < \frac{1}{5\pi}$ then $\Delta_{\zeta,c(t)}$ intersects with $F_{2}$, $F_{3}$ and $F_{4}$ so $\Delta_{\zeta,c(t)}$ is a triangle, 
when $t=\frac{1}{5\pi}$ then $\Delta_{\zeta,c(t)}$ across $(1,0,0)$ a vertex of $\Delta$ and a topological change happens, 
when $\frac{1}{5\pi}< t < \frac{2}{5\pi}$ then $\Delta_{\zeta,c(t)}$ intersects with $F_{1}$, $F_{2}$, $F_{3}$ and $F_{4}$ so $\Delta_{\zeta,c(t)}$ is a square, 
when $t=\frac{2}{5\pi}$ then $\Delta_{\zeta,c(t)}$ across $(0,1,0)$ a vertex of $\Delta$ and a topological change happens, 
when $\frac{2}{5\pi}< t < \frac{1}{2\pi}$ then $\Delta_{\zeta,c(t)}$ intersects with $F_{1}$, $F_{2}$ and $F_{3}$ so $\Delta_{\zeta,c(t)}$ is a triangle, 
and when $t=\frac{1}{2\pi}$ then $\Delta_{\zeta,c(t)}$ is one point $\{(0,0,0)\}$ this means that $\Delta_{\zeta,c(t)}$ vanishes. 
Hence a solution $\{L'_{\zeta,c(t)}\}_{t\in I}$ of a generalized Lagrangian mean curvature flow with singularities and topological changes 
exists for $t\in I=[0,\frac{1}{2\pi})$. It forms singularities and topological changes when $t=\frac{1}{5\pi}$ and $t=\frac{2}{5\pi}$, and vanishes when $t=\frac{1}{2\pi}$. 

One can see the topology of $L_{\zeta,c(t)}=M^{\sigma}_{\zeta,c(t)}\times S^1$ (since now $T_{\zeta}\cong S^1$) by the same argument as explained in the proof of Proposition A.3 in \cite{Yamamoto}. 
In fact the topology of $M^{\sigma}_{\zeta,c(t)}$ is $S^2$ when $0\leq t < \frac{1}{5\pi}$, is $T^2$ when $ \frac{1}{5\pi}< t < \frac{2}{5\pi}$, is $S^2$ when $\frac{2}{5\pi}< t < \frac{1}{2\pi}$. 
\appendix
\section{}\label{app}
In Section \ref{angle}, we introduce the notion of the weighted Hamiltonian stationary. 
In this appendix, we explain the meaning of it. 
Let $(M, \omega, g, J, \Omega)$ be a $2m$-dimensional almost Calabi--Yau manifold with the function $\psi$ defined by (\ref{psi}) and 
$F:L\rightarrow M$ be a Lagrangian immersion with the Lagrangian angle $\theta_{F}$. 
Then we define $F:L\rightarrow M$ is a weighted Hamiltonian stationary if $\Delta_{f}\theta_{F}=0$. 
Here $f$ is a function on $L$ defined by $f:=-mF^{*}\psi$ and $\Delta_{f}$ is the weighted Laplacian on Riemannian manifold $(L,F^{*}g)$ defined by 
$\Delta_{f}u:=\Delta u +\langle\nabla u,\nabla f \rangle$ where $\Delta$ is the standard Laplacian on $L$ with respect to a metric $F^{*}g$. 

Let $\tilde{g}:=e^{2\psi}g$ be a conformal rescaling of $g$ on $M$, then we get a new Riemannian manifold $(M,\tilde{g})$. 
For an immersion $F:L\rightarrow M$, we define a weighted volume functional $\mathrm{Vol}_{\psi}$ by 
$$\mathrm{Vol}_{\psi}(F):=\int_{L}dV_{F^{*}\tilde{g}}$$
where $dV_{F^{*}\tilde{g}}$ is the volume form on $L$ with respect to a metric $F^{*}\tilde{g}$. 
Note that the relation between $dV_{F^{*}\tilde{g}}$ and $dV_{F^{*}g}$ is given by 
$$dV_{F^{*}\tilde{g}}=e^{mF^{*}\psi}dV_{F^{*}g}=e^{-f}dV_{F^{*}g}.$$ 
Then we consider a symplectic manifold $(M,\omega)$ with the weighted volume functional $\mathrm{Vol}_{\psi}$. 
The following proposition is the meaning of the weighted Hamiltonian stationary. 
\begin{proposition}
A Lagrangian immersion $F:L\rightarrow M$ is weighted Hamiltonian stationary if and only if 
$F$ is a critical point of the weighted volume functional $\mathrm{Vol}_{\psi}$ along Hamiltonian deformations with respect to $\omega$. 
\end{proposition}
\begin{proof}
Let $\{F_{t}:L\rightarrow M\}_{t}$ be a Hamiltonian deformation of $F$ with Hamiltonian functions $\{h_{t}:L\rightarrow\mathbb{R}\}_{t}$, that is, $F_{0}=F$ and 
\begin{align}\label{ham}
\omega(\frac{\partial F}{\partial t},\,\bullet\,)=-dh_{t}. 
\end{align}
If $L$ is non-compact we assume that each $h_{t}$ has a compact support. 
Then the first variation of $\mathrm{Vol}_{\psi}$ at $F$ along $\{F_{t}:L\rightarrow M\}_{t}$ is derived by the first variation formula as
\begin{align*}
\frac{d}{dt}\bigg|_{t=0}\mathrm{Vol}_{\psi}(F_{t})&=\frac{d}{dt}\bigg|_{t=0}\int_{L}e^{mF_{t}^{*}\psi}dV_{F_{t}^{*}g}\\
&=-\int_{L}g(e^{mF^{*}\psi}H-me^{mF^{*}\psi}\nabla \psi^{\bot},\frac{\partial F}{\partial t}\bigg|_{t=0})dV_{F^{*}g}\\
&=-\int_{L}g(H-m\nabla \psi^{\bot},\frac{\partial F}{\partial t}\bigg|_{t=0})e^{-f}dV_{F^{*}g}.
\end{align*}
Next we remember the definition of the generalized mean curvature vector filed $K$, see (\ref{defK}), and use the equation (\ref{K}), then we have 
\begin{align*}
-\int_{L}g(H-m\nabla \psi^{\bot},\frac{\partial F}{\partial t}\bigg|_{t=0})e^{-f}dV_{F^{*}g}&=-\int_{L}g(K,\frac{\partial F}{\partial t}\bigg|_{t=0})e^{-f}dV_{F^{*}g}\\
&=-\int_{L}g(J\nabla\theta_{F},\frac{\partial F}{\partial t}\bigg|_{t=0})e^{-f}dV_{F^{*}g}. 
\end{align*}
Since the equation (\ref{ham}) is equivalent to $\frac{\partial F}{\partial t}=J\nabla h_{t}$, we have
\begin{align*}
-\int_{L}g(J\nabla\theta_{F},\frac{\partial F}{\partial t}\bigg|_{t=0})e^{-f}dV_{F^{*}g}&=-\int_{L}g(J\nabla\theta_{F},J\nabla h_{0})e^{-f}dV_{F^{*}g}\\
&=-\int_{L}\langle d\theta_{F},dh_{0}\rangle_{F^{*}g}e^{-f}dV_{F^{*}g}\\
&=-\int_{L}(\Delta_{f}\theta_{F}) h_{0} e^{-f}dV_{F^{*}g}\\
&=-\int_{L}(\Delta_{f}\theta_{F}) h_{0} dV_{F^{*}\tilde{g}}. 
\end{align*}
In the third equality, we use the another definition of $\Delta_{f}u=\delta_{f}(du)$ where $\delta_{f}$ is the formal adjoint of $d$ with respect to a 
weighted measure $e^{-f}dV_{F^{*}g}$. One can easily show that $\delta_{f}(du)=\Delta u+\langle\nabla u,\nabla f\rangle_{F^{*}g}$. 
Now we can take any $h_{0}$ thus it is clear that the first variation of $\mathrm{Vol}_{\psi}$ at $F$ along all Hamiltonian deformations is zero 
if and only if $\Delta_{f}\theta_{F}=0$. 
\end{proof}


\begin{thebibliography}{9}

\bibitem{Behrndt}
T. Behrndt.
\newblock Generalized {L}agrangian mean curvature flow in {K}\"ahler manifolds that are almost {E}instein.
\newblock In {\em Complex and {D}ifferential {G}eometry}, volume~8 of {\em Springer {P}roceedings in {M}athematics}, pages 65--79. Springer-{V}erlag, 2011.

\bibitem{Behrndt2}
T. Behrndt.
\newblock Mean curvature flow of {L}agrangian submanifolds with isolated conical singularities.
\newblock {\em \rm{arXive:1107.4803v1}}, 2011.


\bibitem{Brakke} 
K. A. Brakke. 
\newblock The motion of a surface by its mean curvature. 
\newblock Mathematical Notes, Princeton University Press, 1978.





\bibitem{Guillemin2}
V. Guillemin.
\newblock Moment maps and combinatorial invariants of {H}amiltonian $T^{n}$-spaces.
\newblock {\em Progress in Mathematics}, 122, {\em Birkh\"auser Boston, Inc., Boston}, MA, 1994.


\bibitem{HarveyLawson}
R. Harvey and H.~B. Lawson, Jr.
\newblock Calibrated geometries.
\newblock {\em Acta Math.}, 148:47--157, 1982.


\bibitem{Joyce}
D. Joyce.
\newblock Special {L}agrangian {$m$}-folds in {$\Bbb C\sp m$} with symmetries.
\newblock {\em Duke Math. J.}, 115(1):1--51, 2002.


\bibitem{LeeWang}
Y.~I. Lee and M.-T. Wang.
\newblock Hamiltonian stationary cones and self-similar solutions in higher dimensions.
\newblock {\em Trans. Amer. Math. Soc.}, 362(3):1491--1503, 2010.




\bibitem{Mironov}
A. Mironov.
\newblock On new examples of {H}amiltonian-minimal and minimal Lagrangian submanifolds in $\mathbb{C}^{m}$ and $\mathbb{CP}^{m}$.
\newblock {\em Sb. Math.}, 195(1):85--96, 2004.

\bibitem{MironovPanov}
A. Mironov and T. Panov. 
\newblock Hamiltonian-minimal Lagrangian submanifolds in toric varieties.
\newblock {\em Russ. Math. Surv}, 68(2):392--394, 2013.




\bibitem{StromingerYauZaslow}
A. Strominger, S.-T. Yau, and E. Zaslow.
\newblock Mirror symmetry is {$T$}-duality.
\newblock {\em Nuclear Phys. B}, 479(1-2):243--259, 1996.


\bibitem{Yamamoto}
H. Yamamoto
\newblock Special {L}agrangians and {L}agrangian self-similar solutions in cones over toric {S}saki manifolds. 
\newblock {\em \rm{arXive:1203.3934v2}}, 2013.

\end{thebibliography}
\end{document}